\documentclass{amsproc}

\usepackage{amsmath,amsthm}
\usepackage{graphicx}
\usepackage[usenames,dvipsnames]{color}
\usepackage{tikz} \usepackage{enumerate}
\usetikzlibrary{shapes.geometric, arrows}
\usetikzlibrary{decorations.markings} \usepackage{graphicx}
\usepackage{amsthm} \usepackage{amsmath} \usepackage{caption}
\usepackage{amssymb} \usepackage{amsfonts} \usepackage{float}
\usepackage{faktor} \usetikzlibrary{plotmarks} \usepackage{circuitikz}
\usepackage[utf8]{inputenc} \usepackage[english]{babel}
\DeclareGraphicsExtensions{.jpg,.png,.pdf}
\numberwithin{equation}{section}

\newtheorem{theorem}{Theorem}[section]

\newtheorem{corollary}{Corollary}[section]

\theoremstyle{definition}

\theoremstyle{remark}
\newtheorem{remark}[theorem]{Remark}

\numberwithin{equation}{section}

\begin{document}

\title{Homotopical Complexity of a $3D$ Billiard Flow}


\author{Caleb C. Moxley}
\address{The University of Alabama at Birmingham\\
Department of Mathematics\\
1300 University Blvd., Suite 476A\\
Birmingham, AL 35294}
\email{ccmoxley@uab.edu}

\author{Nandor J. Simanyi}
\address{The University of Alabama at Birmingham\\
Department of Mathematics\\
1300 University Blvd., Suite 490B\\
Birmingham, AL 35294}
\email{simanyi@uab.edu}

\thanks{The second author thankfully acknowledges the support of the National Science Foundation, grant no. 1301537}


\date{\today}

\begin{abstract}
In this paper we study the homotopical rotation vectors and the
homotopical rotation sets for the billiard flow on the unit flat torus
with three, mutually intersecting and mutually orthogonal cylindrical
scatterers removed from it.

The natural habitat for these objects is the infinite cone erected
upon the Cantor set $\text{Ends}(\textbf{F}_3)$ of all ``ends'' of the
hyperbolic group $\textbf{F}_3=\pi_1(\mathbf{Q})$. An element of
$\text{Ends}(\textbf{F}_3)$ describes the direction in (the Cayley
graph of) the group $\textbf{F}_3$ in which the considered trajectory
escapes to infinity, whereas the height function $s$ ($s\ge 0$) of
the cone gives us the average speed at which this escape takes place.

The main results of this paper claim that the orbits can only escape
to infinity at a speed not exceeding $\sqrt{3}$, and in any direction
$e\in\text{Ends}(\mathbf{F}_3)$ the escape is feasible with any
prescribed speed $s$, $0\leq s\leq 1/3$. This means that the radial
upper and lower bounds for the rotation set $R$ are actually pretty
close to each other. Furthermore, we prove the convexity of the set
$AR$ of constructible rotation vectors, and that the set of rotation
vectors of periodic orbits is dense in $AR$. We also provide effective
lower and upper bounds for the topological entropy of the studied
billiard flow.
\end{abstract}

\maketitle

\section{Introduction}
The concept of rotation number finds its origin in the study of the average
rotation around the circle $S^1$ per iteration, as classically defined by H.
Poincar\'e in the 1880's \cite{P(1952)}, when one iterates an orientation-preserving circle
homeomorphism $f:S^1 \rightarrow S^1$. This is equivalent to studying the
average displacement $(1/n)(F^n(x)-x)$ ($x \in \mathbb{R}$) for the iterates
$F^n$ of a lifting $F:\mathbb{R} \rightarrow \mathbb{R}$ of $f$ on the universal
covering space $\mathbb{R}$ of $S^1$. The study of fine homotopical properties of
geodesic lines on negatively curved, closed surfaces goes back at least to
Morse \cite{M(1924)}. As far as we know, the first appearance of the concept of
homological rotation vectors (associated with flows on manifolds) was the
paper of Schwartzman \cite{Sch(1957)}, see also Boyland \cite{B(2000)} for further references
and a good survey of homotopical invariants associated with geodesic flows.
Following an analogous pattern, in \cite{BMS(2006)} we defined the (still commutative)
rotation numbers of a $2D$ billiard flow on the billiard table $\mathbb{T}^2 =
\mathbb{R}^2/\mathbb{Z}^2$ with one convex obstacle (scatterer) $\mathbf{O}$
removed. Thus, the billiard table (configuration space) of the model in
\cite{BMS(2006)} was $\mathbf{Q} = \mathbb{T}^2\setminus\mathbf{O}$.  Technically
speaking, we considered trajectory segments $\{x(t) | 0 \le t \le T\} \subset
\mathbf{Q}$ of the billiard flow, lifted them to the universal covering space
$\mathbb{R}^2$ of $ \mathbb{T}^2$ (not of the configuration space $\mathbf{Q}$),
and then systematically studied the rotation vectors as limiting vectors of
the average displacement $(1/T)(\tilde{x}(T)-\tilde{x}(0)) \in \mathbb{R}^2$
of the lifted orbit segments $\{\tilde{x}(t)|0 \le t \le T\}$ as $T
\rightarrow \infty$. These rotation vectors are still ``commutative'', for
they belong to the vector space $\mathbb{R}^2$.

In this paper we consider the billiard flow on the unit flat torus
$\mathbb{T}^3=\faktor{\mathbb{R}^3}{\mathbb{Z}^3}$ with the three
one-dimensional subtori
\[
S_i=\{(x_1,x_2,x_3)\in \mathbb{T}^3: x_j=0 \text{ for all } j\neq i \},\text { \hspace{0.1cm}} (i=1,2,3),
\]
serving as scatterers.

Despite all the advantages of the homological
(or ``commutative'') rotation vectors (i. e. that they belong to a
real vector space, and this provides us with useful tools to construct
actual trajectories with prescribed rotational behaviour), in our
current view the ``right'' lifting of the trajectory segments
$\{x(t)|0 \le t \le T\} \subset \mathbf{Q}$ is to lift these segments
to the universal covering space of $\mathbf{Q}$, not of $\mathbb{T}^3$. This, in
turn, causes a profound difference in the nature of the arising
rotation ``numbers'', primarily because the fundamental group
$\pi_1(\mathbf{Q})$ of the configuration space $\mathbf{Q}$ is the
highly complex group $\textbf{F}_3$ freely generated by three generators
(see \S2 below or \cite{M(1991)}). After a bounded modification,
trajectory segments $\{x(t)| 0 \le t \le T\} \subset \mathbf{Q}$ give
rise to closed loops $\gamma_T$ in $\mathbf{Q}$, thus defining an
element $g_T = [\gamma_T]$ in the fundamental group
$\pi_1(\mathbf{Q}) = \textbf{F}_3$. The limiting behavior of $g_T$ as
$T \rightarrow \infty$ will be investigated, quite naturally, from two
viewpoints:

\begin{enumerate}
   \item The direction ``$e$'' is to be determined, in which the element $g_T$
   escapes to infinity in the hyperbolic group $\textbf{F}_3$ or, equivalently,
   in its Cayley graph $\mathbf{G}$, see \S2 below. All possible
   directions $e$ form the horizon or the so called ideal boundary
   $\text{Ends}(\textbf{F}_3)$ of the group $\textbf{F}_3=
   \pi_1(\mathbf{Q})$, see \cite{CP(1993)}.
   \item The average speed $s = \lim_{T \rightarrow \infty}
   (1/T)\text{dist}(g_T, 1)$ is to be determined, at which the element
   $g_T$ escapes to infinity, as $T \rightarrow \infty$. 
   These limits (or limits $\lim_{T_n \rightarrow \infty} 
   (1/T_n)\text{dist}(g_{T_n}, 1)$ for sequences of positive reals $T_n \nearrow \infty$) are nonnegative real numbers.
\end{enumerate}
The natural habitat for the two limit data $(s,e)$ is the infinite cone
\begin{displaymath}
C=([0, \infty) \times \text{Ends}(\textbf{F}_3))/(\{0\} \times \text{Ends}(\textbf{F}_3))
\end{displaymath}
erected upon the set $\text{Ends}(\textbf{F}_3)$, the latter supplied
with the usual Cantor space topology. Since the homotopical ``rotation
vectors'' $(s,e) \in C$ (and the corresponding homotopical rotation
sets) are defined in terms of the non-commutative fundamental group
$\pi_1(\mathbf{Q}) = \textbf{F}_3$, these notions will be justifiably
called homotopical or noncommutative rotation numbers and sets.

The rotation set arising from trajectories obtained by the arc-length
minimizing variational method will be the so called admissible
homotopical rotation set $AR \subset C$. The homotopical rotation set
$R$ defined without the restriction of admissibility will be denoted
by $R$. Plainly, $AR \subset R$ and these sets are closed subsets of
the cone $C$.

The main results of this paper are Theorems 3.1--3.4. Theorem 3.2
claims that the set $R$ is contained in the closed ball
$B(0,\,\sqrt{3})$ of radius $\sqrt{3}$ centered at the vertex
$0=\{0\}\times\text{Ends}(\textbf{F}_3)$ of the cone $C$. In
particular, both sets $AR$ and $R$ are compact. Theorem 3.1 asserts
that the set $AR$ contains the closed ball $B(0, 1/3)$ of $C$.
Furthermore, in Theorem 3.3 we prove the convexity of the set $AR$,
and Theorem 3.4 says that the set of rotation vectors of admissible
periodic orbits is dense in $AR$. We also provide effective lower and
upper bounds for the topological entropy of the studied billiard flow.
Thus, these two results provide a pretty detailed description of the
homotopical complexity of billiard orbits: Any direction
$e\in\text{Ends}(\mathbf{F}_3)$ is feasible for the trajectory to go
to infinity, the speed of escape $s$ cannot be bigger than $\sqrt{3}$,
whereas any speed $s$, $0\leq s\leq 1/3$, is achievable in any
direction $e\in\text{Ends}(\mathbf{F}_3)$.

In \S4 we compare our results, Theorems 1 and 2, with the
homotopical complexity of geodesic flows on strictly negatively curved,
compact manifolds, and observe some striking difference. Finally, in the
closing section we prove the lower bound $(1/3)\ln 5=0.536479\dots$
and the upper bound $2\sqrt{3}\ln 12=8.607696\dots$ for the topological
entropy of our billiard flow.

\section{Prerequisites. Model and Geometry of Orbits}

In this paper we are studying the homotopical properties of the trajectories of the following
billiard flow $(\mathbf{M}, \{S^t\}, \mu)$: From the standard flat $3$-torus
$\mathbb{T}^3=\faktor{\mathbb{R}^3}{\mathbb{Z}^3}$ we cut out the open, tubular $r_0$-neighborhoods
($r_0>0$ is small enough) of the three one-dimensional subtori
\[
T_i=\{(x_1,x_2,x_3)\in \mathbb{T}^3: x_j=0 \text{ for all } j\neq i \},\text { \hspace{0.1cm}} (i=1,2,3),
\]
serving as scatterers. In the resulting configuration space $\mathbf{Q}=\mathbf{Q}_{r_0}$ a point is moving
uniformly with unit speed, bouncing back at the piecewise smooth boundary $\partial\mathbf{Q}$ of
$\mathbf{Q}$ according to the law of specular reflections. The natural invariant measure (Liouville measure)
$\mu$ of the resulting Hamiltonian flow $(\mathbf{M}, \{S^t\}, \mu)$ can be obtained by normalizing the product
of the Lebesgue measure of $\mathbf{Q}$ and the hypersurface measure of the unit sphere $S^2$ of velocities.

A fundamental domain $\Delta_0$ of the configuration space $\mathbf{Q}$ can be obtained by taking
\[
\Delta_0=\left\{x=(x_1,x_2,x_3)\in [0,1]^3 \big\| \text{dist}(x,T_i)\ge r_0,\quad i=1,2,3\right\}
\]
by glueing together the opposite faces
\[
F_i^0=\left\{(x_1,x_2,x_3)\in \Delta_0\big\| x_i=0 \right\}
\]
and
\[
F_i^1=\left\{(x_1,x_2,x_3)\in \Delta_0\big\| x_i=1 \right\}.
\]
This shows that the space $\mathbf{Q}$ is homotopically equivalent to the bouquet of three circles and, henceforth,
the fundamental group $G=\pi_1(\mathbf{Q})$ is the group $\mathbf{F}_3=\mathbf{F}_3(a,b,c)$ freely generated by the elements $a$, $b$,
and $c$, where $a$ corresponds to crossing the face $F_1^0$ (or $F_1^1$) in the positive direction, and $b$, $c$, are
defined analogously for the remaining two coordinate directions.

Consequently, the Cayley-graph $\Gamma$ of $\pi_1(\mathbf{Q})=\mathbf{F}_3=\mathbf{F}_3(a,b,c)$ is the full $6$-regular tree on the vertex set
set $\mathbf{F}_3$: From every element $g$ of $\mathbf{F}_3$ there emanate $6$ edges in the directions of $ga$, $ga^{-1}$, $gb$, $gb^{-1}$,
$gc$, and $gc^{-1}$, respectively, see \cite{BH(1999)}. The graph $\Gamma$ is a rooted tree with the identity element
$1\in \mathbf{F}_3$ as the root.

We are going to study the asymptotic (in the long time run) homotopical properties of orbit segments $S^{[0,T]}x$ of our
billiard flow, where $T\to \infty$. Given any infinite sequence $S^{[0,T_n]}x_n$ of orbit segments with $T_n\to\infty$,
by adding a bounded curve to the beginning and ending parts of these orbit segments, we may assume that
$q(S^{T_n}x_n)=q(x_n)=q_0\in \mathbf{Q}$ ($n=1,2,\dots$) is a fixed base point $q_0$ for the fundamental group
$\pi_1(\mathbf{Q},q_0)=\mathbf{F}_3$. The loops
\[
\left\{q(S^tx_n)\big\| 0\le t\le T_n\right\}
\]
naturally give rise to the curves
\[
\gamma_n=\left\{\gamma_n(t)\big\| 0\le t\le T_n\right\}\subset\Gamma
\]
with $\gamma_n(0)=1$ (the root of $\Gamma$). We are interested in describing all possible pairs $(s,w)$ of limiting speeds
\[
s=\lim_{n\to\infty}T_n^{-1}\cdot\text{dist}(\gamma_n(T_n), e)
\]
and directions $e\in\text{Ends}(\Gamma)$ in which the curves $\gamma_n$ go to infinity in $\Gamma$. Here $0\le s<\infty$, and
$w$ is an arbitrary element of the Cantor set $\text{Ends}(\Gamma)$ of all ends of the free group $\mathbf{F}_3$, see \cite{BH(1999)}.
So the natural habitat of the (set of) limiting homotopical ``rotation vectors'' $(s,e)$ is the infinite cone
\[
C=[0, \infty)\times\text{Ends}(\Gamma)/\{0\}\times\text{Ends}(\Gamma)
\]
erected upon the Cantor set $\text{Ends}(\Gamma)$. For convenience, we identify all homotopical rotation vectors $(0,e)$
with zero speed. The arising set of all achievable homotopical rotation vectors $(s,e)\in C$ will be called the (full) rotation
set and denoted by $R$.

\subsection{Principles for the design of admissible orbits}

\medskip

1. The trajectories enter each elementary cell (cube) $C$ through an ``entrance face'', and leave it through an ``exit face'',
the latter being different from the former one.

2. The orbit visits (bounces back from) two or three different (tubular $r_0$-neighborhoods of) edges of the elementary
cell $C$ during its stay in that cell. The first edge to be visited is an edge of the entrance face, whereas the last one is an edge of the
exit face.

3. During the flight in $C$, for any two consecutive edges $E_1$ and $E_2$ it should be true that the intersection of the convex hull
$\text{conv}(E_1, E_2)$ of $E_1$ and $E_2$ with the cube $C$ has a non-empty interior in $\mathbb{R}^3$.

4. At each bouncing back from (the tubular $r_0$-neighborhoods of) an edge $E$ there is a pre-assigned ``past force'', arising from the already
constructed past of the trajectory, that pulls the point of contact with (the tubular $r_0$-neighborhoods of) $E$ toward one end of $E$. The further
construction of the admissible orbit should be such that the force originating from the future part of the orbit should pull the point of contact
with $E$ toward the other end of $E$, hence keeping a balance to avoid the orbit, to be constructed by the arc-length minimizing variational
method, reaching a corner of the phase space at one end of the edge $E$, and thus runs into a singularity.

5. In the actual construction of admissible orbit segments (see \S3
below), while using the arch-length minimizing variational method for
the construction, we will be assuming that $r_0=0$. Strictly speaking,
the variational method then does not yield a genuine cylindrical
billiard trajectory, but this problem can easily be overcome by
``swelling'' the degenerate cylinders (with $r_0=0$) to realistic ones
with $r_0>0$ and applying a continuity argument.

\medskip

\begin{remark}
  Property $4$  above will guarantee that the point of contact with $E$ will be an interior point of $E$ after applying the length minimizing
  variational method.
\end{remark}

\section{The admissible rotation set}

As was stated in \S 2, we consider the billiard flow in the unit flat
torus $\mathbb{T}^3=\faktor{\mathbb{R}^3}{\mathbb{Z}^3}$ with the
three one-dimensional subtori
\[S_i=\{(x_1,x_2,x_3)\in \mathbb{T}^3: x_j=0 \text{ for all } j\neq i \},\text { \hspace{0.1cm}} (i=1,2,3),\]
serving as scatterers. The admissible trajectories will be constructed
--- by using the variational method --- in such a way that two forces
will guarantee that the point of contact $p_n$ with any scatterer
$S_i$ is an ``interior'' point of $S_i$, i.e. it is different from the
corner point (0,0,0). This will be ensured employing two forces: one
toward the point of contact $p_{n-1}$ pulling $p_n$ in $S_i$ in one
direction and another one pointing toward $p_{n+1}$ and pulling $p_n$
in $S_i$ in the other direction. The admissible trajectories to be
constructed below will be lifted to the covering space $\mathbb{R}^3$
of $\mathbb{T}^3$ right away during the construction.
  
\subsection{The turn $ab$}

\noindent 
In the discussion of each turn below, the reference compartment in
which the turn takes place is the unit cube
$C(\frac{1}{2},\frac{1}{2},\frac{1}{2})$ centered at
$(\frac{1}{2},\frac{1}{2},\frac{1}{2})$. In this case, the trajectory
enters $C(\frac{1}{2},\frac{1}{2},\frac{1}{2})$ from
$C(-\frac{1}{2},\frac{1}{2},\frac{1}{2})$ and exits
$C(\frac{1}{2},\frac{1}{2},\frac{1}{2})$ toward
$C(\frac{1}{2},\frac{3}{2},\frac{1}{2})$. The underlying philosophy of
our construction is that the past is pre-determined and the future is
to be constructed. For the entry edge (the scatterer edge that
corresponds to the reflection at the time of crossing the common face
of $C(\frac{1}{2},\frac{1}{2},\frac{1}{2})$ and
$C(-\frac{1}{2},\frac{1}{2},\frac{1}{2})$) there are four
possibilities: $\{0\}\times\{1\}\times[0,1]$,
$\{0\}\times\{0\}\times[0,1]$, $\{0\}\times[0,1]\times\{0\}$, and
$\{0\}\times[0,1]\times\{1\}$. Of these four choices, the last two are
isomorphic, thus we do not consider the last one. \\ \\
\textit{Case 3.1.1} The entry edge is $\{0\}\times\{1\}\times[0,1]$. By symmetry,
we may assume that the force pre-determined by the past pulls the
contact point with this edge downward, i.e. it wants to decrease the
third coordinate $x_3$ of the point of contact $(0,1,x_3)$. Now we
design the next two edges of reflection as $[0,1]\times\{0\}\times\{1\}$,
then $\{1\}\times\{1\}\times[0,1] $.
Elementary geometry shows that the time the orbit spends in the
compartment is less than 3 units. Furthermore, the force acting in the
second segment of the just constructed orbit piece in
$C(\frac{1}{2},\frac{1}{2},\frac{1}{2})$ pulls the point of contact
$(1,1,x_3)$ with the exit edge upward, i.e. it wants to increase
$x_3$.  \\ \\ \textit{Case 3.1.2} The entry edge is
$\{0\}\times\{0\}\times[0,1]$. By symmetry again, we may assume that
the force pre-determined by the past pulls the contact point with this
edge downward. Now the next edge of contact is designed as
$[0,1]\times\{1\}\times\{1\}$, serving immediately as the exit edge
from the compartment $C(\frac{1}{2},\frac{1}{2},\frac{1}{2})$. Observe
that the time spend is less than $\sqrt{3}$, and the past force pulls
the exit edge point of contact $(x_1,1,1)$ downward, i.e. it wants to
decrease $x_1$.  \\ \\ \textit{Case 3.1.3.a} The entry edge is
$\{0\}\times[0,1]\times\{0\}$, and the past force pulls the value
$x_2$ of the contact point $(0,x_2,0)$ downward. Again, we construct
the exit edge from $C(\frac{1}{2},\frac{1}{2},\frac{1}{2})$ right
after the entry edge as $\{1\}\times\{1\}\times[0,1]$ --- though the
edge $[0,1]\times\{1\}\times\{1\}$ would also suffice --- and observe
that the time the orbit spends in the compartment is less than
$\sqrt{3}$. The past force is pulling the value $x_3$
of the exit point of contact $(1,1,x_3)$ downward. \\ \\
\textit{Case 3.1.3.b} The entry edge is
$\{0\}\times[0,1]\times\{0\}$, and the past force pulls the value
$x_2$ of the contact point $(0,x_2,0)$ upward. Now, the next two edges
of contact will be $[0,1]\times\{0\}\times\{1\}$ and
$\{1\}\times\{1\}\times[0,1]$, the latter edge being the exit
edge. Observe that the time the orbit spends in the reference
compartment is less than 3. The past force is pulling the value $x_3$
of the exit point of contact $(1,1,x_3)$ upward.  \\ \\ This completes
the list of cases in \S3.1, i.e. when the turn is $ab$. All other
turns $\epsilon\delta$ ($\epsilon\ne\delta^{-1}$) are isomorphic to the case $ab$, where
$\epsilon,\delta\in\{a,a^{-1},b,b^{-1},c,c^{-1}\}=\mathcal{G}$, and
$\epsilon\delta$ is not equal to the square of any element of
$\mathcal{G}$. The remaining case, up to isomorphism, is the passage
straight through the reference compartment, i.e. $a^2$, traveling in
the $x_1$ direction.

\subsection{The straight-through passage $a^2$}
The orbit segment to be constructed enters \newline
$C(\frac{1}{2},\frac{1}{2},\frac{1}{2})$ from the cube
$C(-\frac{1}{2},\frac{1}{2},\frac{1}{2})$ and exits to
$C(\frac{3}{2},\frac{1}{2},\frac{1}{2})$. By symmetry, we may assume
that the entry edge is $\{0\}\times[0,1]\times\{0\}$ with the force
predetermined pulling $x_2$ of the entry contact point downward. We
choose the exit edge to be $\{1\}\times\{1\}\times [0,1]$, following
immediately the entry edge. Observe that the orbit segments spends a
time less than $\sqrt{3}$ in the reference compartment
$C(\frac{1}{2},\frac{1}{2},\frac{1}{2})$ and that the past force pulls
the coordinate $x_3$ of the exit contact point of $(1,1,x_3)$
downward.

\medskip

\subsection{Anchoring}
In order for the above construction to work and produce a finite
admissible orbit segment, we need to add two additional cylinders or
edges to the segment, the two anchors: one at the beginning of the
segment, and one at the end. We will require that the initial point of
the segment (to be constructed with minimal arc-length) be the
midpoint of the added initial anchor edge, whereas the last point of
the segment be the midpoint of the added terminal anchor edge. We
construct these two anchors in such a way that they provide the needed
balance with the future or past force, respectively.

Here we describe the simple construction of the terminal anchor; the
construction of the initial anchor is analogous.  We may assume that
the trajectory enters the compartment
$C(\frac{1}{2},\frac{1}{2},\frac{1}{2})$ from
$C(-\frac{1}{2},\frac{1}{2},\frac{1}{2})$, and the entry edge is
$\{0\}\times\{1\}\times[0,1]$. By symmetry, we may assume that the
force pre-determined by the past pulls the contact point with this
edge downward, i.e. it wants to decrease the third coordinate $x_3$ of
the point of contact $(0,1,x_3)$.  We declare now the terminal anchor
the edge $\{1\}\times[0,1]\times\{1\}$ of the compartment
$C(\frac{1}{2},\frac{1}{2},\frac{1}{2})$.

\medskip

\begin{theorem}[Corollary to the admissible construction]
The admissible rotation set contains the ball centered at 0 with radius $\frac{1}{3}$, i.e. 
\[B\left(0,\frac{1}{3}\right)\subset AR\]
\end{theorem}

\begin{proof}
Let $\underline{w}=w_1w_2w_3\dots$ be an infinite word in $\mathbf{F}_3(a,b,c)$
--- the group freely generated by three elements --- corresponding to
a given end of the hyperbolic group
$\mathbf{F}_3(a,b,c)=\pi_1(\mathbf{Q})$. The results of this section allow us
to construct an infinitely long admissible orbit $S^{[0,\infty)}x_0$
  that follows the itinerary $\underline{w}$ and spends at most time 3
  in each elementary cell. This means that
\[\underset{n\to\infty}{\lim\inf} \frac{n}{T_n}\geq \frac{1}{3},\]
where the itinerary of $S^{[0,T_n]}x_0$ is $w_1w_2\dots w_n$ and $T_n$
is the time of exiting the $n^{\text{th}}$ cell on
$S^{[0,\infty)}x_0$. Since the speed of the orbit can be decreased
  arbitrarily by injecting an appropriate amount of idle time in
  $S^{[0,\infty)}x_0$, we see that every homotopical rotation
    ``vector'' $(s,e)\in B(0,\frac{1}{3})$ may be obtained as an
    admissible rotation vector.
\end{proof}

\begin{theorem}[]
The the full rotation set, and thus the admissible rotation set, is contained in the ball centered at 0 with radius $\sqrt{3}$, i.e. 
\[AR\subset R\subset B\left(0,\sqrt{3}\right).\]
\end{theorem}

\begin{proof}
Consider an orbit segment $S^{[0,T]}x_0$ with a large value of $T$ ---
eventually we will allow $T$ to tend to infinity and make asymptotic
estimates. Denote by $n_x$, $n_y$, $n_z$ the number of $y-z$, $z-x$,
and $x-y$ face crossings on $S^{[0,T]}x_0$. Since the integral of
$\lvert v_1(t) \rvert$ between two $y-z$ face crossings is at least
one, we get
\[\overset{T}{\underset{0}{\int}} \lvert v_1(t) \rvert dt \geq n_x-1. \]
Similarly, 
\[\overset{T}{\underset{0}{\int}} \lvert v_2(t) \rvert dt \geq n_y-1 \] 
and 
\[\overset{T}{\underset{0}{\int}} \lvert v_3(t) \rvert dt \geq n_z-1. \]
Adding these inequalities, we get 
\[T=\overset{T}{\underset{0}{\int}} \lvert v(t) \rvert dt \geq \frac{1}{\sqrt{3}}\overset{T}{\underset{0}{\int}}
(\lvert v_1(t) \rvert + \lvert v_2(t) \rvert + \lvert v_3(t) \rvert) dt\geq \frac{n_x+n_y+n_z-3}{\sqrt{3}}, \]
that is 
\[\underset{T\to\infty}{\lim\sup}\frac{n_x+n_y+n_z}{T}\leq \sqrt{3}. \]
\end{proof}

\begin{theorem}[Convexity of the Admissible Rotation Set]
The admissible rotation set $AR$ is a convex subset of the cone $C$.
\end{theorem}

\begin{proof}
The cone $C$ is a totally disconnected, Cantor set-type family of
infinite rays that are glued together at their common endpoint, the
vertex of the cone. Therefore, the convexity of $AR$ means that for
any $(s,e)\in AR$ and for any $t$ with $0\le t\le s$ we have $(t,e)\in
AR$. However, this immediately follows from our construction, since we
can always insert a suitable amount of idle runs into an admissible
orbit segment to be constructed, hence slowing it down to the
asymptotic speed $t$, as required.
\end{proof}

\begin{theorem}[Periodic Rotation Vectors are Dense in $AR$]
All the rotation vectors $(s,e)\in AR$ that correspond to periodic
admissible trajectories form a dense subset of $AR$.
\end{theorem}

\begin{proof}
The following statement immediately follows from the flexibility of our construction:
Given any finite, admissible trajectory segment $S^{[0,T]}x_0$, with the approximative
prescribed rotation vector $(s,e)\in AR$, one can always append a bounded initial
and terminal segment to $S^{[0,T]}x_0$, so that after this expansion the following properties hold:

\begin{enumerate}
\item[$(1)$] The initial and the terminal compartments of $S^{[0,T]}x_0$ differ by the
  same integer translation vector $\vec{v}\in\mathbb{Z}^3$ by which the initial and
  terminal anchor edges differ;
\item[$(2)$] The future force acting on the point of contact with the initial anchor
  is opposite to the past force acting on the point of contact with the terminal anchor edge.
\end{enumerate}

These two properties gurantee that, by releasing the midpoints as the points of contact and just
requiring that they differ by the integer vector $\vec{v}$, one constructs a periodic admissible orbit
with with the approximative rotation vector $(s,e)$.
\end{proof}

\medskip

\section{Comparing our results with geodesic flows}

The completely hyperbolic, semi-dispersive billiard flows are widely considered, with justice, as models analogous to the
geodesic flows on negatively curved, closed manifolds. Consider, therefore, a smooth, compact, connected Riemannian manifold
$M=M^n$ with everywhere strictly negative sectional curvatures and empty boundary. Let $\mathcal{T}_1 M$ be the unit tangent bundle
of $M$, and $(\mathcal{T}_1 M, \{S^t\}, \mu)$ the arising geodesic flow on $M$.

The fundamental group $\Gamma=\pi_1(M)$ is known to have only one end in the Freudenthal compactification sense,
$\text{Ends}(\Gamma)=\{\partial\Gamma\}$, where $\partial\Gamma$ is the so called Gromov boundary or ideal boundary of the
hyperbolic group $\Gamma$, see \cite{BH(1999)}. The ideal boundary $\partial \Gamma$ possesses a natural topology, introduced by Gromov,
that makes it naturally diffeomorphic to the infinite horizon $S^{n-1}$ of the universal covering Hadamard space $\tilde{M}$
of $M=M^n$, see again \cite{BH(1999)}. Let $p:\; \tilde{M}\to M$ be the universal covering map of $M$, and choose base points
$y_0\in\tilde{M}$, $x_0=p(y_0)$. Consider the embedding $\alpha:\; \Gamma \to \tilde{M}$ by lifting the $x_0$-loops to $\tilde{M}$
so that the starting point $x_0$ is lifted to $y_0$, and the end point will be the $\alpha$-image of the element of $\Gamma$
represented by the lifted loop. According to the Svarc-Milnor Lemma \cite{BH(1999)}, the map $\alpha$ is a quasi-isometry, that is,
\[
d(\alpha(g), \alpha(h))\le C_1d(g,h)+C_2,
\]
and
\[
d(g,h)\le C_3 d(\alpha(g), \alpha(h))+C_4
\]
for all $g,\, h\in\Gamma$ with fixed positive constants $C_i$. Consequently, for this model the relevant cone
$C$, containing the homotopical rotation vectors, is
\[
C=[0,\infty)\times \partial\Gamma/\{0\}\times \partial\Gamma
=[0,\infty)\times S^{n-1}/\{0\}\times S^{n-1},
\]
and the full homotopical rotation set $R$ lies between two concentric spheres centered at the vertex $0$ of the cone $C$, i.e.
\[
R\subset\overline{B}(0,r_2)\setminus B(0.r_1)
\]
with $0<r_1<r_2$. This is in sharp contrast with the homotopical rotation set $R$ of our model: The latter one
is actually a neighborhood of the vertex $0$. Even the smaller set $AR$ turns out to be a neighborhood of $0$.
This is explained by the fact that in the billiard model there are tools to slow down (reduce the speed $s$)
the admissible trajectories by inserting in them a sufficient amount of idle runs.

\section{Topological entropy of the flow}

The corollary below is a direct byproduct of our upper bound estimate
for the full homotopical rotation set in the previous section and
provides a positive constant as the upper estimate for the topological
entropy $h_{top}(r_0)$ of our $3D$ billiard flow with three mutually
intersecting and perpendicular scatterers.

\begin{corollary} The topological entropy $h_{top}(r_0)$ for the $3D$ billiard flow studied in this paper is bounded above in the following way. 
\[h_{top}(r_0)\leq 2\sqrt{3}\ln 12= 8.607696\dots\]
\end{corollary}

\begin{proof} \hspace{0.1cm} \\ 

\noindent
We define a partition of the configuration space of our period
billiard table by dividing it into seven pairwise almost-disjoint
domains $\mathcal{D}_0$, $\mathcal{D}_1^{\pm}$, $\mathcal{D}_2^{\pm}$,
and $\mathcal{D}_3^{\pm}$ depicted in Figure 5.1.1. We define
$\mathcal{D}_k^{+}$ ($k=1,2,3$) to be all points $(x_i)_{i=1}^3\in
\mathbf{Q}$ for which the fractional part $\{x_{k}\}$ of the
coordinate $x_{k}$ is no more than $\varepsilon_0$ for some fixed,
small $\varepsilon_0>0$.  Similarly, we define $\mathcal{D}_k^{-}$
($k=1,2,3$) to be all points $(x_i)_{i=1}^3\in \mathbf{Q}$ for which
the fractional part $\{1-x_{k}\}$ of $1-x_{k}$ is no more than
$\varepsilon_0$. Finally, we define $D_0$ to be the closure of the set
\[\mathbf{Q}\setminus \left(\left(\overset{3}{\underset{k=1}\cup} \mathcal{D}_k^+\right) \cup \left(\overset{3}{\underset{k=1}\cup} \mathcal{D}_k^-\right)\right).\]
Because the union of these seven domains is $\mathbf{Q}$ and because
this union is an almost disjoint one --- since the domains only
intersect at their piecewise planar boundaries --- we have that
$\mathbf{Q}= \mathcal{D}_0\cup \left(
\left(\overset{3}{\underset{k=1}\cup} \mathcal{D}_k^+\right) \cup
\left(\overset{3}{\underset{k=1}\cup} \mathcal{D}_k^-\right) \right)$
is a partition of $\mathbf{Q}$ from a dynamical viewpoint. We call
this partition $\Pi$ .
 
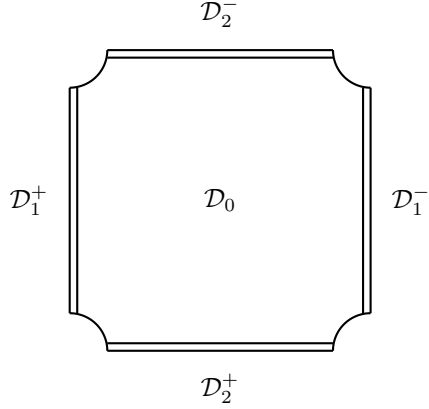
\begin{figure}[H]
\centering
\begin{tikzpicture}
\draw[black,thick] (0.5,4) -- (3.5,4); 
\draw[black,thick] (0.5,0) -- (3.5,0); 
\draw[black,thick] (4,0.5) -- (4,3.5); 
\draw[black,thick] (0,0.5) -- (0,3.5); 
\draw[black,thick] (0.5,3.9) -- (3.5,3.9); 
\draw[black,thick] (0.5,0.1) -- (3.5,0.1); 
\draw[black,thick] (3.9,0.5) -- (3.9,3.5); 
\draw[black,thick] (0.1,0.5) -- (0.1,3.5); 
\node[] at (2,2) {$\mathcal{D}_0$};
\node[] at (4.54,2) {$\mathcal{D}_1^-$};
\node[] at (-0.54,2) {$\mathcal{D}_1^+$};
\node[] at (2,-0.5) {$\mathcal{D}_2^+$};
\node[] at (2,4.5) {$\mathcal{D}_2^-$};
\draw [black, thick,domain=0:90] plot ({0.5*cos(\x)}, {0.5*sin(\x)});
\draw [black, thick,domain=0:90] plot ({-0.5*cos(\x)+4}, {-0.5*sin(\x)+4});
\draw [black, thick,domain=-90:0] plot ({-0.5*cos(\x)+4}, {-0.5*sin(\x)});
\draw [black, thick,domain=90:180] plot ({-0.5*cos(\x)}, {-0.5*sin(\x)+4});
\end{tikzpicture}
\caption*{Figure 5.1.1 - A Cross Section of Partition $\Pi$}
\end{figure}
 
  \noindent
We wish to show that the supremum (i.e. the coarsest common
refinement) \[ \bigvee_{n=-\infty}^{\infty}
S^{-n\varepsilon_0}(\Pi) \] of the partitions
$S^{-n\varepsilon_0}(\Pi)$ is the trivial partition into singletons,
modulo the zero-measured sets --- that is, we wish to show that $\Pi$
is a generating partition. To see this, assume that two phase points
$x=(q_1,v_1)$ and $y=(q_2,v_2)$ (where $q_1,q_2\in\mathbf{Q}$,
$v_1,v_2\in\mathbb{R}^3$, and $\lvert\lvert v_i \rvert\rvert =1$) have
identical symbolic future $\Pi$-itineraries, recorded at
$n\varepsilon_0$ moments of time. Then elementary inspection shows
that the two phase points $S^{\tau}y$ and $x$ belong to the same local
stable curve, where $\tau\in \mathbb{R}$ is a time-synchronizing
constant. Shared past $\Pi$-itineraries yield the same result for the
unstable curves. Thus, $S^{\tau}y=x$ for some $\tau$ whenever $x$ and
$y$ have identical $\Pi$-itineraries in both time directions. That is,
$x=y$ for a typical pair $(x,y)$, so $\Pi$ is a generating partition.
\\

\noindent
For any time $T>0$ (which will eventually tend to infinity), denote by
$N(T)$ the number of all possible $\Pi$-itineraries of trajectory
segments $S^{[0,T]}x$, $x\in \mathbf{M}$. Because $\Pi$ is a
generating partition, we have that
\begin{equation*}h_{top}(r_0)=\underset{T\to\infty} \lim \frac{1}{T} \ln N(T).\end{equation*}
Clearly, any orbit segment $\{x(t): 0 \leq t \leq T \}$ alternates
between the domains $\mathcal{D}_0$ and $\mathcal{D}^* =
\overline{\mathbf{Q}\setminus\mathcal{D}_0}$. Consider the orbit
segment $x=\{ x(t): 0 \leq t \leq T \}$ lifted to the covering space
$\mathbb{R}^3$ of $\mathbb{T}^3$. Let $t_1$ be the time when $x$
leaves the domain $\mathcal{D}_1^+$ ($D_1^-$), and let $t_2$ be the
next time when $x$ re-enters the same domain. Note: $0\leq t_1<t_2\leq T$. \\
From the proof of Theorem 3.2, we see that
\[ \int_{t_1}^{t_2} \lvert \dot{x}_1(t) \rvert dt \geq 1-\varepsilon_0.\]
Therefore, the number of times the orbit segment $x$ visits the domains $\mathcal{D}_1^+$ ($\mathcal{D}_1^-$) is at most 
\[ \frac{1}{1-\varepsilon_0}\int_{0}^{T} \lvert \dot{x}_1(t) \rvert dt +1. \]
Applying this upper estimate to $\mathcal{D}_1^+$ and
$\mathcal{D}_1^-$ and then the analogous upper estimates for the
number of visits to $\mathcal{D}_2^{\pm}$ and $\mathcal{D}_3^{\pm}$
and, finally, taking the sum of the arising six estimates, we get
that the total number of visits by $x$ to the six domains
$\mathcal{D}_1^{\pm}$, $\mathcal{D}_2^{\pm}$, and
$\mathcal{D}_3^{\pm}$ is at most
\[ \frac{2}{1-\varepsilon_0}\int_{0}^{T} \left( \lvert \dot{x}_1(t) \rvert + \lvert \dot{x}_2(t) \rvert+ \lvert \dot{x}_3(t) \rvert \right) dt
+ 6 \leq \frac{2\sqrt{3}T}{1-\varepsilon_0}+6. \]
Now, since $x$ alternates between $\mathcal{D}_0$ and the union of the other six domains, the total number of times $x$ visits $\mathcal{D}_0$ is at most 
\[f(T,\varepsilon_0) := \frac{2\sqrt{3}T}{1-\varepsilon_0}+7.\]
After the orbit segment enters any of the domains
$\mathcal{D}_1^{\pm}$, $\mathcal{D}_2^{\pm}$, $\mathcal{D}_3^{\pm}$,
it has exactly two sides of the domain (i.e. two combinatorial
possibilities) to exit it, and once it enters $\mathcal{D}_0$, it has
six sides to exit it.

\noindent
The above arguments immediately give the upper estimate for $N(T)$: 
\[N(T)\leq 12^{f(T,\varepsilon_0)}, \] which is the number $N(T)$ of all possible
symbolic types of orbit segments of length $T$. Thus, using the
generating property of $\Pi$, we may calculate an upper estimate for
$h_{top}(r_0)$ by taking the natural logarithm of our upper estimate
for $N(T)$, dividing by $T$, and passing to the limit as $T\to\infty$
and then as $\varepsilon_0 \to 0$.

\end{proof}

\noindent
The next corollary results from our lower estimate for the admissible
homotopical rotation set --- and hence for the full homotopical
rotation set.
\begin{corollary} The topological entropy $h_{top}(r_0)$ for the $3D$ billiard flow studied in this paper is bounded below in the following way. 
\[0.536479 \dots= \frac{\ln 5}{3} \leq  h_{top}(r_0)\]
\end{corollary}

\begin{proof} \hspace{0.1cm} \\ 
Theorem 3.1 says that the words $w(\{x(t): 0\leq t \leq
T\})$ corresponding to all orbits $\{x(t): 0\leq t \leq T\}$ of length
$T$ fill the ball of radius $T/3$ in the Cayley graph of the group
$\mathbf{F}_3$, which is a 6-regular tree. Thus, the number of different
homotopy types of these orbits $\{x(t): 0\leq t \leq T\}$ is at least
$k\cdot 5^{\frac{T}{3}}$, where $k$ is some constant. After taking the
natural logarithm of this lower estimate, dividing by $T$, and passing
to the limit as $T\to\infty$, we get the claim of the corollary.
\end{proof}

\medskip

\noindent
Acknowledgement. Special thanks are due to Gabor Moussong for his valuable suggestions and comments.

\bibliographystyle{amsalpha}

\end{document}